\documentclass[12pt]{article}

\usepackage{amsmath, amsthm, amsfonts, amssymb}
\usepackage[dvipdfm]{pict2e}

\usepackage[dvipdfm, dvips]{graphicx, color}

\theoremstyle{break}

\title{Affine twist deformation of a sphere with holes}
\author{Takayuki Masuda\footnote{Department of Mathematics, Graduate School of Science, Osaka University,
Toyonaka, Osaka 560-0043, Japan,
\texttt{t-masuda@cr.sci.math.osaka-u.ac.jp}}}
\date{}

\newtheorem{theorem}{Theorem}[section]
\newtheorem{proposition}[theorem]{Proposition}
\newtheorem{lemma}[theorem]{Lemma}

\newtheorem{remark}[theorem]{Remark}
\newtheorem{example}[theorem]{Example}
\newtheorem{definition}[theorem]{Definition}

\begin{document}
\maketitle

\section*{Abstract}
In this paper, we introduce a new parameter, the {\it affine twist parameter} for the affine deformation of a sphere with holes. We show that the affine deformation space can be parametrized by Margulis invariants and affine twist parameters. The affine twist parameter is canonically regarded as a correspondence to the Fenchel-Nielsen twist parameter in Teichmuller theory.

\section{Introduction}
Let $S_{b+1} (b \geq 3)$ denote a sphere with $(b+1)$-boundaries. We fix a Fuchsian holonomy $\pi_1(S_{b+1}, pt) \to {\rm PSL}(2, {\mathbb R})$, which sends all peripheral curves to hyperbolic elements. The image of the holonomy is a free group of rank $b$, which is denoted by $G_b$. 
An {\it affine deformation} is a faithful representation $G_b \to \Gamma_b \subset SO^0(2,1) \ltimes {\mathbb R}^2_1$, which is defined by ${\rm PSL}(2, {\mathbb R}) \cong SO^0(2,1)$ and a cocycle ${\bf u}$. The cocycle is a map $G_b \to {\mathbb R}^2_1$. The space of cocycles is identified with the cohomology class ${\rm H}^1(G_b, {\mathbb R}_1^2)$. 
The group $\Gamma_b$ naturally acts a $3$-dimensional Minkowski space $E_1^2$ isometrically, and is called an {\it affine transformation group}. 
The property whether the affine action is properly discontinuous and free depends on the cocycle. We are interested in the set ${\bf Proper}_b$ in ${\rm H}^1 (G_b, {\mathbb R}_1^2)$, in which consists of cocycles that make affine transformation groups which act properly discontinuously on $E_1^2$.
\newline

When $b=2$ (a pair of pants) or, in general, a free group of rank two, the affine deformations of the Fuchsian group $G$ were deeply studied by V.Charette T.Drumm and W.Goldman. (See $\cite{cdg2010} \cite{cdg2012} \cite{cdg2014}$.) In their works, {\it Margulis invariants} have been used to parametrize ${\rm H}^1(G, {\mathbb R}_1^2)$. The Margulis invariant was first introduced by G.Margulis $\cite{m1983}$ in order to study the properly discontinuity of the affine transformation groups.
However if $b \geq 3$, it is difficult to handle ${\rm H}^1(G_b, {\mathbb R}_1^2)$ by using only Margulis invariants. 

Here we introduce the {\it affine twist parameters}; Consider a pants decomposition on $S_{b+1}$, which consists of $(b-2)$-dividing simple closed curves in $S_{b+1}$. We associate an affine twist parameter to each dividing curve. Together with Margulis invariants for dividing curves and boundary components, we get a $(3b-3)$-dimensional linear space $D_b$.

The purpose of this paper is to show that the parameter space $D_b$ canonically corresponds to the Fenchel-Nielsen parameter in Teichmuller theory. 
We first show the following.
\begin{theorem}
\label{Linear isomorphism}
There is a canonical isomorphism between $D_b$ and ${\rm H}^1(G_b, {\mathbb R}_1^2)$.
\end{theorem}

The affine twist parameter $t_k$ has a relation with a cocycle ${\rm AT}_k$. We indicate that the cocycle ${\rm AT}_k$ is corresponding to the Fenchel-Nielsen twist; The Lorentzian space ${\mathbb R}_1^2$ is isometric to a Lie algebra $sl_2({\mathbb R})$. 
Using the correspondence, Goldman and Margulis \cite{gm2000} give a correspondence between the cocycle and the deformation of the hyperbolic structures of $G$.
It is the correspondence between the Margulis invariant and the infinitesimal deformation $\frac{d}{dt} L_g^{{\bf u}} (0)$ of a displacement length of an element $g$ of $G$. (See $\S \ref{relation to deformations of the hyperbolic structures}$.) 

We will show the following equation:
\begin{theorem}
\label{wolpert goldman margulis}
Let $h_l$ be a dividing curve in the pants decomposition of $S_{b+1}$, 
and $f_l$ be a simple closed curve which intersects $h_l$ twice and disjoint from the other dividing curves. Then
\begin{eqnarray*}
\frac{1}{2} \frac{d L_{f_l}^{{\bf u}}}{dt} (0) = \alpha_{{\bf u}_0}(f_l) + t_l (\cos{\theta_l} + \cos{\theta_l'})
\end{eqnarray*}
holds, for ${\bf u} = {\bf u}_0 + \sum_{k=1}^{b-2} t_k {\rm AT}_k$, where $\theta_l , \theta_l'$ are angles between $h_l$ and $f_l$ in $S_{b+1}$ at each of the two intersections.
\end{theorem}
Theorem $\ref{wolpert goldman margulis}$ indicates that the affine twist (parameter), in our sense, correspondents to the Fenchel-Nielsen twist, in the sense of Goldman and Margulis as above discussion: Theorem $\ref{wolpert goldman margulis}$ has a similarity with the work by S.Wolpert, where he showed the correspondence between the twists and the angles by closed curves. $(\cite{w1981}.)$
\newline

Finally we remark a classification for elements in ${\bf Proper}_b$. 
The author finds a part of ${\bf Proper}_b$ in $D_b$.
In order to find the elements in ${\bf Proper}_b$, we use {\it crooked planes}.
They were used by Charette, Drumm and Goldman to classify the properly discontinuous affine transformations $(\cite{cdg2010} \cite{cdg2012} \cite{cdg2014})$.
The way they classified them is to assign disjoint crooked planes.
The disjointness of crooked planes implies that the corresponding representation acts properly discontinuously on $E_1^2$.
The author applies the way for the affine deformations of $S_{b+1}$. Then the author finds the part of ${\bf Proper}_b$. Furthermore we can find a relation between the crooked planes and the affine twists.
\newline

In \S \ref{notation}, we will explain basic theories of two geometries; a hyperbolic geometry of the sphere with $(b+1)$-holes and a geometry of the Lorentzian space-time. In \S \ref{affine deformations}, we will explain affine deformations of the sphere with $(b+1)$-holes, and proof two important lemmas for $\S \ref{Linear space of affine deformations}$. In $\S \ref{Linear space of affine deformations}$, we will proof Theorem $\ref{Linear isomorphism}$. In \S \ref{relation to deformations of the hyperbolic structures}, we will refer to the relation between the cocycles and the deformations of the hyperbolic structures, and then we will proof Theorem $\ref{wolpert goldman margulis}$. In \S \ref{properly discontinuous action}, we will remark the part of ${\bf Proper}_b$.

\section*{Acknowledgment}
The author is grateful to Professor Hideki Miyachi for many important advice.
Furthermore the author would like to thank Professor Todd Drumm for many interesting discussion about this paper, as well as Professor William Goldman and Professor Suhyoung Choi for helpful advice at MSRI.

\section{Notation}
\label{notation}
\subsection{Hyperbolic geometry of punctured spheres}
\label{hyperbolic geometry}
For $b \geq 3$, let $S_{b+1} :=S_{(0, b+1)}$ be a sphere with $(b+1)$-boundaries. We fix a faithful representation (we call a {\it holonomy}) $\rho_0$: $\pi_1(S_{b+1}, pt) \to {\rm PSL}(2, {\mathbb R})$. ({\it i.e.} fix the hyperbolic structure.) We always identify ${\rm PSL(2, {\mathbb R})}$ with $SO^0(2,1)$. Hence ${\mathbb H}^2 / G_b \cong S_{b+1}$. 
Note that the group $G_b$ is a free group of rank $b$. Thus it denotes: 
\begin{eqnarray}
G_b = \langle g_1, g_2, \ldots, g_{b}, g_{b+1} \mid g_1 \cdot g_2 \cdots g_b \cdot g_{b+1} = id \rangle,
\end{eqnarray}
where each generator $g_i$ corresponds to a boundary component of $S_{b+1}$. We denote the set of this indexes of the generators by ${\mathbb I} :=\{ 1, 2, \ldots, b+1 \}$. (See Figure $\ref{hyperbolic disk}$.)

\begin{figure}[htbp] 
\begin{center}
\begin{tabular}{c}
\begin{minipage}{0.5\hsize}
\begin{center}
\includegraphics[width=6cm]{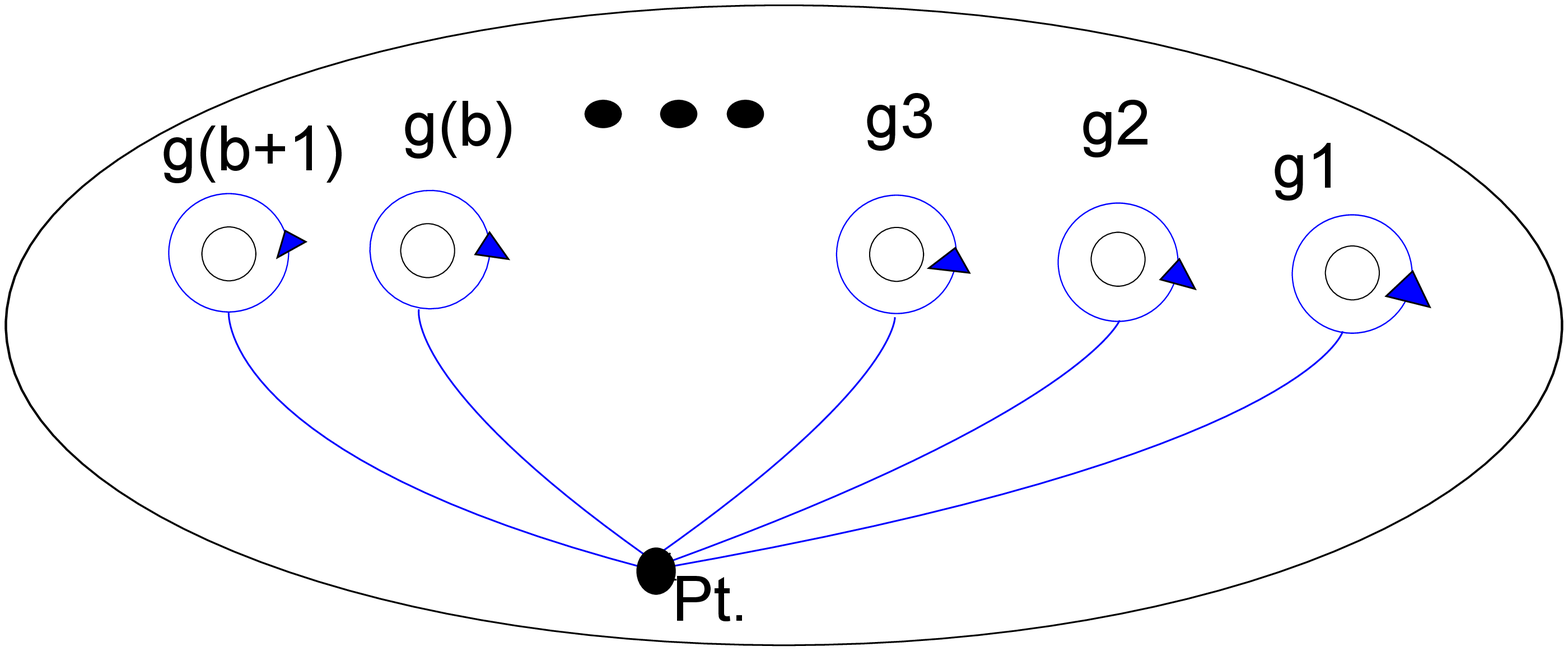} 
\hspace{1.6cm}
\end{center}
\end{minipage}
\begin{minipage}{0.33\hsize}
\begin{center}
\includegraphics[width=4cm]{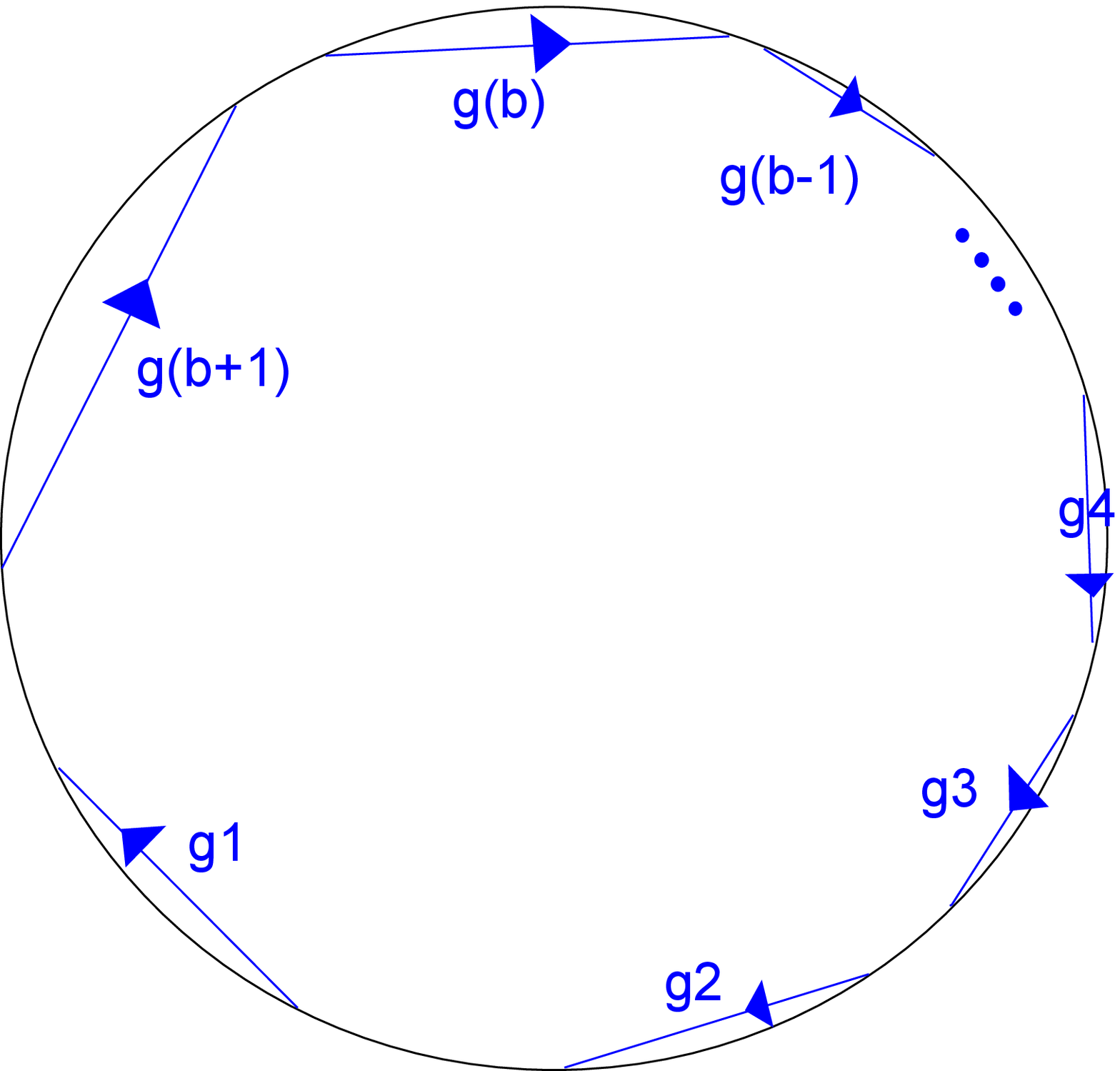} 
\hspace{1.6cm}
\end{center}
\end{minipage}
\end{tabular}
\caption{Developing $S_{b+1}$ into ${\mathbb H}^2$} 
\label{hyperbolic disk} 
\end{center}
\end{figure} 

\subsubsection{Fenchel-Neilsen coordinates}
Here we mainly focus on decomposing $S_{b+1}$ into pairs of pants. 
For $j \in {\mathbb J} := \{ 1, 2, \ldots, b-2 \}$, we define
\begin{eqnarray}
h_j := g_{j+1}^{-1} \cdot g_{j}^{-1} \cdots g_1^{-1}.
\end{eqnarray}
Each $h_j$ corresponds to a dividing curve when $S_{b+1}$ is decomposed like Figure $\ref{pants decomposition}$. The subgroup $P_j < G_b$ is defined by:
\begin{eqnarray}
P_1 &:=& \langle g_1, g_2, h_1 \mid g_1 \cdot g_2 \cdot h_1 = id \rangle \\
P_j &:=& \langle h_{j-1}, g_{j+1}, h_j \mid h_{j-1}^{-1} \cdot g_{j+1} \cdot h_j = id \rangle, \,\,\, 2 \leq j \leq b-2 \\
P_{b-1} &:=& \langle h_{b-2}, g_b, g_{b+1} \mid h_{b-2}^{-1} \cdot g_b \cdot g_{b+1} = id \rangle.
\end{eqnarray}
Then we can have the $(b-1)$-pairs of pants ${\mathbb H}^2 / P_j$ from ${\mathbb H}^2 / G_b$. We also call each $P_j$ a {\it pair of pants}. 

\begin{figure}[htbp] 
\begin{center}
\begin{tabular}{c}
\begin{minipage}{0.5\hsize}
\begin{center}
\includegraphics[width=6cm]{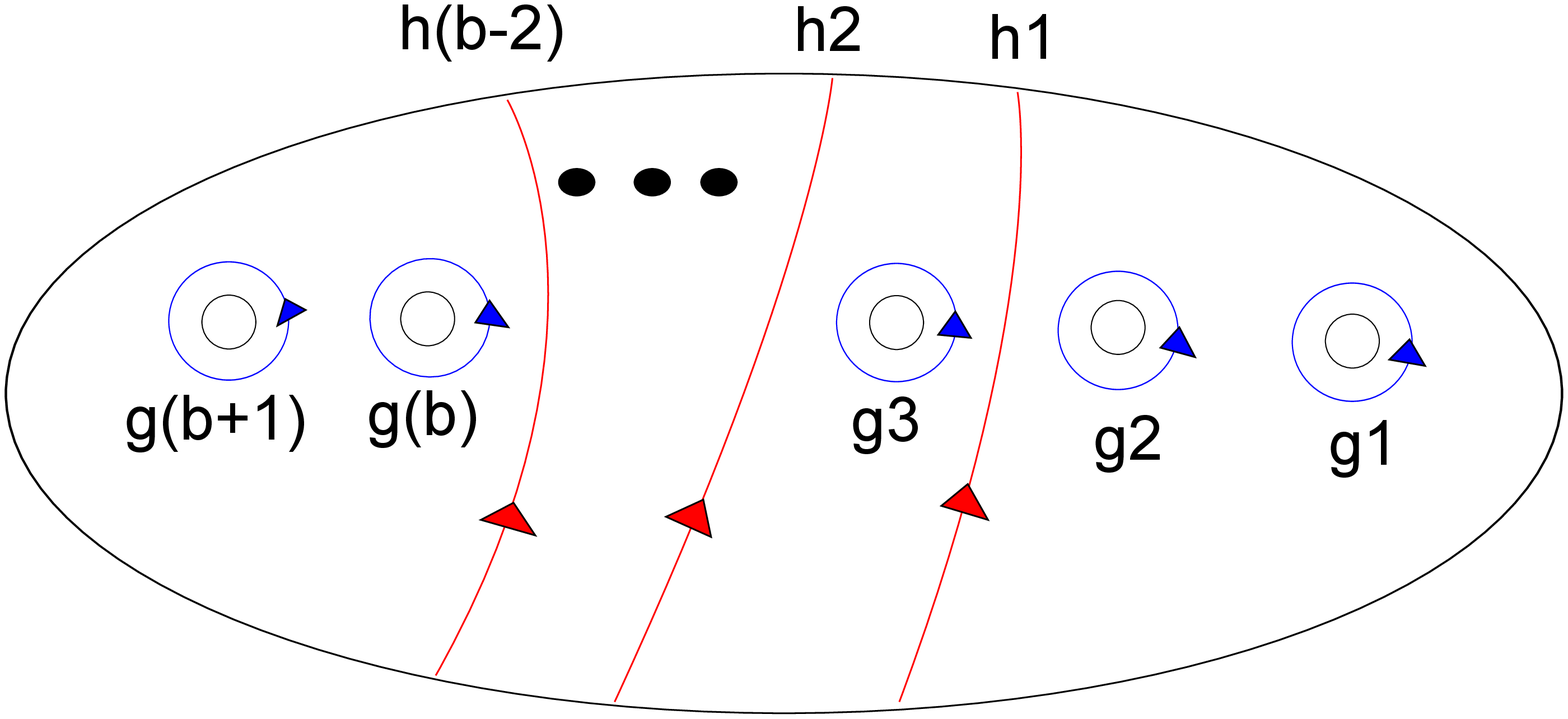} 
\hspace{1.6cm}
\end{center}
\end{minipage}
\begin{minipage}{0.33\hsize}
\begin{center}
\includegraphics[width=4cm]{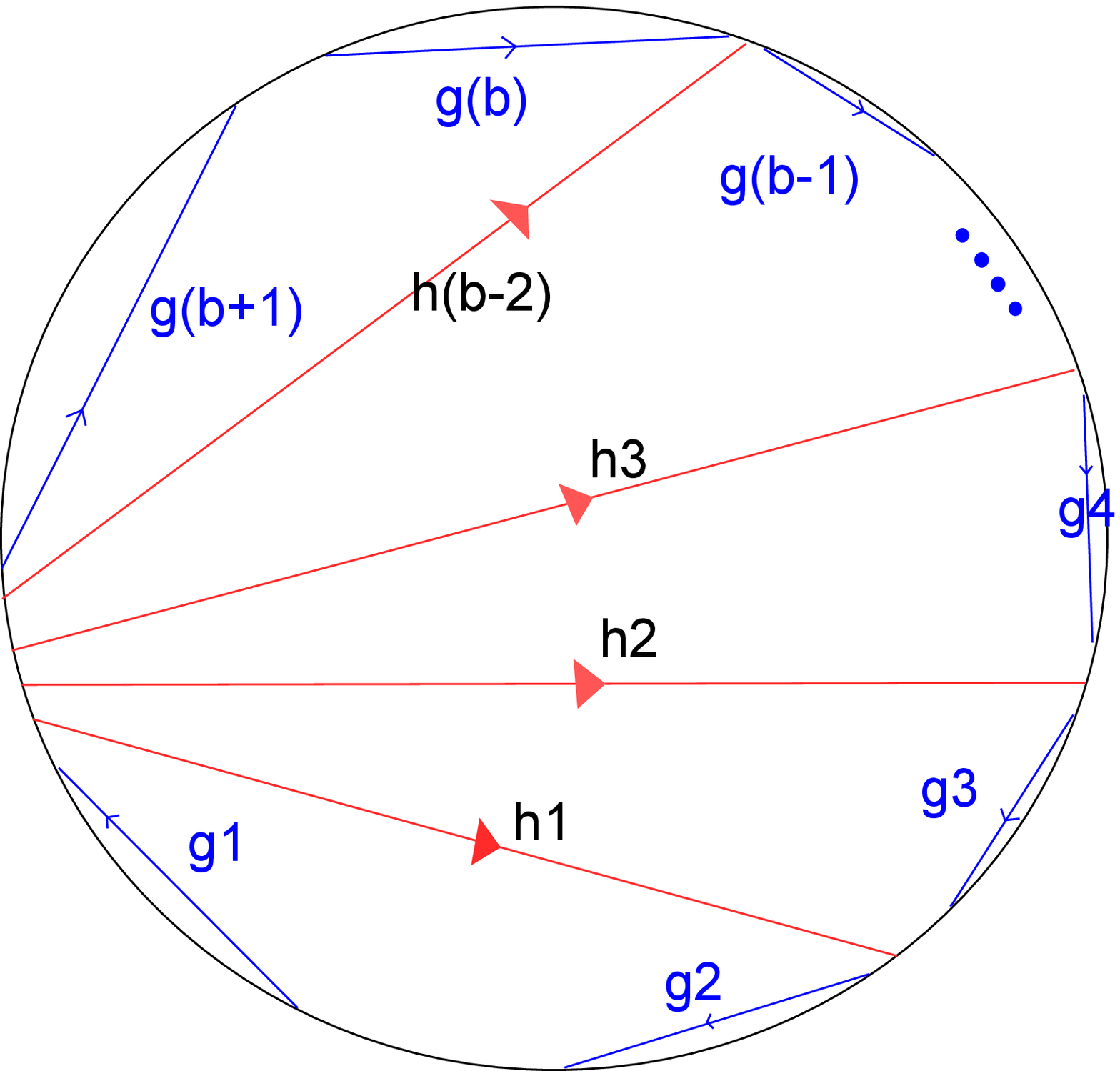} 
\hspace{1.6cm}
\end{center}
\end{minipage}
\end{tabular}
\caption{Pants decomposition}
\label{pants decomposition} 
\end{center}
\end{figure} 

\begin{remark}[Fenchel-Nielsen]
\label{fenchelnielsen}
The Teichmuller space of a topological sphere with $(b+1)$-holes is to be :
\begin{eqnarray*}
{\mathbb R}^{3b-3} \cong  {\mathbb R}^{2b-3} \times {\mathbb R}^{b-1}.
\end{eqnarray*}
Here the former part of the direct product represents the lengths of geodesics corresponding to dividing curves and the boundaries, and the latter does the twists along the dividing curves. 
\end{remark}

\subsection{Lorentzian Geometry}
\subsubsection{Minkowski space-time}
The {\it Minkowski space-time} is an affine space $E_1^2$ with the underlying space ${\mathbb R}_1^2$, where ${\mathbb R}_1^2$ is a three dimensional vector space with the {\it Lorentzian inner product} $B({\bf x}, {\bf y}) = x_1 y_1 + x_2 y_2 -x_3 y_3$ for ${\bf x}, {\bf y} \in {\mathbb R}_1^2$. The following definitions are given in $\cite{dg1990} \cite{cdg2010}$. 
\begin{definition}
\label{lorentzian vector product}
The {\rm Lorentzian vector product} with respect to $B$ is the map $\boxtimes:{\mathbb R}_1^2 \times {\mathbb R}_1^2 \to {\mathbb R}_1^2$, which satisfies the following equations:
For ${\bf x}, {\bf y}, {\bf z}, {\bf w} \in {\mathbb R}_1^2$,
\begin{itemize}
\item[$(1)$] $\det{({\bf x}, {\bf y}, {\bf z})} = B({\bf x} \boxtimes {\bf y}, {\bf z})$,
\item[$(2)$] $B({\bf x} \boxtimes {\bf y}, {\bf z} \boxtimes {\bf w}) = B({\bf x}, {\bf w}) B({\bf y}, {\bf z}) - B({\bf x}, {\bf z}) B({\bf y}, {\bf w})$.
\end{itemize}
\end{definition}

The {\it light cone} $C$ is defined by $B({\bf x}, {\bf x}) = 0$ in ${\mathbb R}_1^2$.
 A vector in $C$ is called {\it null} or {\it lightlike}.
\begin{definition}
${\bf x} \in {\mathbb R}_1^2$ is called a {\rm spacelike vector} if $B({\bf x}, {\bf x})>0$,
 and a {\rm timelike vector} if  $B({\bf x}, {\bf x})<0$.
\end{definition}
The interior of the light cone $C$ has two connected components; {\it future-pointing} and {\it past-pointing}. The set of future-pointing unit timelike vectors is an open disk in ${\mathbb R}_1^2$, which has the Klein-Poincare hyperbolic structure induced from the inner product $B$ (See \cite{cdg2010} for detail.).

\subsubsection{Isometries of Minkowski space-time}
We denote the affine transformation group $SO(2,1)^0 \ltimes {\mathbb R}_1^2$ by ${\rm Isom}^0(E_1^2)$. The group ${\rm Isom}^0(E_1^2)$ preserves orientation of ${\mathbb R}_1^2$ and time-orientation. Every element $\gamma$ of ${\rm Isom}^0(E_1^2)$ is represented by $(g, {\bf u}(g))$, where $g$ is in $SO^0(2, 1)$ and ${\bf u}(g)$ is in ${\mathbb R}_1^2$. An element $g$ is called {\it hyperbolic} if $g$ has three different eigenvalues. We choose three normalized eigenvectors as follows:
\begin{enumerate}
\item[$(1)$] ${\bf X}_g^-$ has $\lambda_g$ as eigenvalue, $B({\bf X}_g^-, {\bf X}_g^-)=0$ and the Euclidean length is $1$,
\item[$(2)$] ${\bf X}_g^+$ has $\lambda_g^{-1}$ as eigenvalue, $B({\bf X}_g^+, {\bf X}_g^+)=0$ and the Euclidean length is $1$,
\item[$(3)$] ${\bf X}_g^0$ has $1$ as eigenvalue, $B({\bf X}_g^0, {\bf X}_g^0)=1$,
and $\det{({\bf X}_g^0, {\bf X}_g^-, {\bf X}_g^+)} > 0$,
\end{enumerate}
where  $0<\lambda<1$. Note that $\langle {\bf X}_g^-, {\bf X}_g^+ \rangle = ({\bf X} _g^0)^{\perp}$ with respect to $B$. A transformation $\gamma$ is also called {\it hyperbolic} if $g$ is hyperbolic. 
Note that these eigenvectors have the following relation:
\begin{lemma}
\begin{eqnarray}
\label{parallel}
{\bf X}_g^- \boxtimes {\bf X}_g^+ = -B({\bf X}_g^-, {\bf X}_g^+) {\bf X}_g^0
\end{eqnarray}
holds.
\end{lemma}
\begin{proof}
Since $\langle {\bf X}_g^-, {\bf X}_g^+ \rangle = ({\bf X} _g^0)^{\perp}$, ${\bf X}_g^- \boxtimes {\bf X}_g^+$ is parallel to ${\bf X}_g^0$. By definition $\ref{lorentzian vector product}$, $B({\bf X}_g^- \boxtimes {\bf X}_g^+, {\bf X}_g^- \boxtimes {\bf X}_g^+) = B({\bf X}_g^-, {\bf X}_g^+)^2$. Since $B({\bf X}_g^-, {\bf X}_g^+)$ is negative and $\det{({\bf X}_g^0, {\bf X}_g^-, {\bf X}_g^+)}$ is positive, the equation $\eqref{parallel}$ holds.
\end{proof}

The following lemma gives a relation between angles in the hyperbolic geometry and in the Lorentzian geometry.
\begin{lemma}
\label{angle}
Let $g, h$ be hyperbolic elements in $SO(2,1)^0$. Suppose that the unique invariant lines which $g, h$ have in ${\mathbb H}^2$ are crossing. We can define an angle $\theta$ between the tangent vectors of $g, h$ at the intersection. Then
\begin{eqnarray}
\label{angle and product}
B({\bf X}_g^0, {\bf X}_h^0) = \cos{\theta}
\end{eqnarray}
holds.
\end{lemma}
\begin{proof}
Under a conjugation, we can consider the intersection as $(0, 0, 1)$.
Then we can regard ${\bf X}_g^0$ as $(1, 0, 0)$, and ${\bf X}_h^0$ as $(\cos{\theta}, \sin{\theta}, 0)$. By a direct calculation, we have $\eqref{angle and product}$.
\end{proof}

\subsubsection{Consistently oriented condition}
We took the special generators $g_i$ of $G_b$ in the section $\ref{hyperbolic geometry}$. They have the following properties:
\begin{eqnarray}
\label{consistently oriented}
&B({\bf X}_{g_m}^0, {\bf X}_{g_n}^0)& < -1, \\
&B({\bf X}_{g_m}^0, {\bf X}_{g_n}^{\pm})& < 0,
\end{eqnarray}
where $m \neq n \in {\mathbb I}$. This condition for the set of generators are called the {\it consistently oriented} condition  in $\cite{cdg2010}$. This idea plays an important role in this paper.
To be convenient, we define the notation:
${\bf X}_i^0 := {\bf X}_{g_i}^0, {\bf X}_i^{\pm} := {\bf X}_{g_i}^{\pm}$ for any $i \in {\mathbb I}$, and ${\bf Y}_j^0 :={\bf X}_{h_j}^0, {\bf Y}_j^{\pm} :={\bf X}_{h_j}^{\pm}$ for any $j \in {\mathbb J}$.

\subsubsection{Margulis invariants}
If a hyperbolic element $\gamma=(g, {\bf u}(g))$ in ${\rm Isom}^0(E_1^2)$ acts freely on $E_1^2$, it has a unique invariant line $C_{\gamma}$ in $E_1^2$. On $C_{\gamma}$, $\gamma$ acts as translation. The distance with respect to $B$ is called the {\it Margulis invariant} $\alpha(g)$. This invariant is defined by $\alpha(g):=B(\gamma (x) - x, {\bf X}_{g}^0)$ for any $x \in E_1^2$. (See $\cite{m1983}$.) Then the value of cocycle is represented as:
\begin{eqnarray}
{\bf u}(g) = \alpha(g) {\bf X}_g^0 + c^- {\bf X}_g^+ + c^+ {\bf X}_g^+,
\end{eqnarray}
where $c^{\pm}$ are some real numbers.
The important properties are:
\begin{lemma}[$\cite{dg2001} \cite{cd2009}$]
The Margulis invariants determine the isometry group in ${\rm Isom}^0(E_1^2)$ up to translation. 
\end{lemma}
\begin{lemma}[$\cite{m1983}$]
$\Gamma \subset {\rm Isom}(E_1^2)$ does not act properly discontinuously if there exist two elements $(g, {\bf u}(g)), (h, {\bf u}(h)) \in \Gamma$ such that $\alpha(g) \cdot \alpha(h) < 0$.
\end{lemma}

\section{Affine deformations}
\label{affine deformations}
In this section, we will consider affine deformation groups $\Gamma_b$ of $G_b$.

\subsection{Affine deformations of a sphere with holes}
The homomorphism $\rho:G_b \hookrightarrow {\rm Isom}^0(E_1^2)$ is an {\it affine deformation} if it satisfies the relation $L \circ \rho = id$, where $L:{\rm Isom}^0(E_1^2) \to SO^0(2,1)$ is a projection. Then, for any element $g \in G_b$, we can denote the {\it affine deformation} $\rho(g) = (g, {\bf u}(g))$. The map ${\bf u}:G_b \to {\mathbb R}_1^2$ is called a {\it cocycle}. It satisfies a {\it cocycle condition}: ${\bf u}(gh) = g{\bf u}(h) + {\bf u}(g)$, $(g, h \in G_b)$. A classification of the affine deformations is equivalent to a  classification of cocycles. 
A {\it coboundary} $\delta_{{\bf v}}$ is a cocycle which forms $\delta_{{\bf v}} (g) = {\bf v} - g {\bf v} \in \langle {\bf v} \rangle^{\perp}$ for a vector ${\bf v} \in {\mathbb R}_1^2$ $(g \in G_b)$.
The coboundary $\delta_{{\bf v}}$ corresponds to the translation by ${\bf v}$. We consider the cocycles up to translation, therefore we consider the quotient space ${\rm H}^1 (G_b, {\mathbb R}_1^2)$.
\begin{remark}
\label{dimension}
The linear space ${\rm H}^1 (G_b, {\mathbb R}_1^2)$ is of $(3b-3)$-dimension.
\end{remark}
To be convenient, We denote $\gamma_i:=\rho(g_i)$ and $\eta_j:=\rho(h_j)$. Because $\rho$ is homomorphic, $\eta_j := \gamma_{j+1}^{-1} \cdot \gamma_{j}^{-1} \cdots \gamma_1^{-1}$ holds.
In general, $\alpha(h_j)$ is not always positive even if all $\alpha(g_i)$ are positive. See $\cite{cdg2010} \cite{c2003} \cite{c2006} \cite{c2009}$.

On the pair of pants, we prove the following lemmas.
\begin{lemma}
\label{cocycle on pant}
For every hyperbolic pair of pants $P=\langle f_1, f_2, f_3 \mid f_1 \cdot f_2 \cdot f_3 = id \rangle$ with no cusp, we consider its affine deformations. Set the coefficients ${\bf u}(f_i)=\alpha_i {\bf X}_i^0 + c_i^- {\bf X}_i^- + c_i^+ {\bf X}_i^+ \, \, (i = 1, 2, 3)$. Then a map 
$$
{\mathbb R}^6 \ni (\alpha_1, \alpha_2, \alpha_3, c_1^-, c_1^+, c_2^-) \mapsto (c_2^+, c_3^-, c_3^+) \in {\mathbb R}^3
$$
is a surjective linear map. Furthermore, a map 
$$
{\mathbb R}^6 \ni (\alpha_1, \alpha_2, \alpha_3, c_1^-, c_1^+, c_2^-) \mapsto {\bf u} \in {\rm Z}^1 (G_b, {\mathbb R}_1^2)
$$ is also a linear map.
\end{lemma}

\begin{proof}
The cocycle condition says that ${\bf u}(f_1 \cdot f_2 \cdot f_3) = 0$; namely,
\begin{eqnarray}
f_3^{-1} {\bf u}(f_3) + f_1 {\bf u}(f_2) + {\bf u}(f_1) = 0.
\end{eqnarray}
It implies that
\begin{eqnarray*}
\alpha_3 {\bf X}_3^0 + \lambda_3^{-1} c_3^- {\bf X}_3^- + \lambda_3 c_3^+ {\bf X}_3^+
+ \alpha_2 f_1 {\bf X}_2^0 + c_2^- f_1 {\bf X}_2^- + c_2^+ f_1 {\bf X}_2^+ \\
+ \alpha_1 {\bf X}_1^0 + c_1^- {\bf X}_1^- + c_1^+ {\bf X}_1^+ = 0 .
\end{eqnarray*}
By considering the Lorentzian inner product with ${\bf X}_1^0, {\bf X}_2^0$ and ${\bf X}_3^0$, we have
\begin{eqnarray*}
&(1)& \alpha_1 + \alpha_2 B({\bf X}_1^0, {\bf X}_2^0) + \alpha_3 B({\bf X}_1^0, {\bf X}_3^0) = \\
&&-\lambda_3^{-1} c_3^- B({\bf X}_1^0, {\bf X}_3^-) - \lambda_3 c_3^+ B({\bf X}_1^0, {\bf X}_3^+) - c_2^- B({\bf X}_1^0,  {\bf X}_2^-) - c_2^+ B({\bf X}_1^0, {\bf X}_2^+) \\
&(2)& \alpha_1 B({\bf X}_2^0, {\bf X}_1^0) + \alpha_2 + \alpha_3 B({\bf X}_2^0, {\bf X}_3^0) = \\
&& -\lambda_3^{-1} c_3^- B({\bf X}_2^0, {\bf X}_3^-) - \lambda_3 c_3^+ B({\bf X}_2^0, {\bf X}_3^+) - c_1^- B({\bf X}_2^0, {\bf X}_1^-) - c_1^+ B({\bf X}_2^0, {\bf X}_1^+) \\
&(3)& \alpha_1 B({\bf X}_3^0, {\bf X}_1^0) + \alpha_2  B({\bf X}_3^0, {\bf X}_2^0) + \alpha_3 = \\
&&- c_2^- B({\bf X}_3^0, {\bf X}_2^-) - c_2^+ B({\bf X}_3^0, {\bf X}_2^+) - c_1^- B({\bf X}_3^0, {\bf X}_1^-) - c_1^+ B({\bf X}_3^0, {\bf X}_1^+).
\end{eqnarray*}
We define three matrices by:
\begin{eqnarray*}
A &:=&
\left[
\begin{array}{ccc}
1 & B({\bf X}_1^0, {\bf X}_2^0) & B({\bf X}_1^0, {\bf X}_3^0) \\
B({\bf X}_1^0, {\bf X}_2^0) & 1 & B({\bf X}_2^0, {\bf X}_3^0) \\
B({\bf X}_1^0, {\bf X}_3^0) & B({\bf X}_2^0, {\bf X}_3^0) & 1
\end{array}
\right], \\
B &:=&
\left[
\begin{array}{ccc}
0 & 0 & B({\bf X}_1^0, {\bf X}_2^-) \\
B({\bf X}_2^0, {\bf X}_1^-) & B({\bf X}_2^0, {\bf X}_1^+) & 0 \\
B({\bf X}_3^0, {\bf X}_1^-) & B({\bf X}_3^0, {\bf X}_1^+) & B({\bf X}_3^0, {\bf X}_2^-)
\end{array}
\right], \\
C &:=&
\left[
\begin{array}{ccc}
B({\bf X}_1^0, {\bf X}_2^+) & \lambda_3^{-1} B({\bf X}_1^0, {\bf X}_3^-) & \lambda_3 B({\bf X}_1^0, {\bf X}_3^+) \\
0 & \lambda_3^{-1} B({\bf X}_2^0, {\bf X}_3^-) & \lambda_3 B({\bf X}_2^0, {\bf X}_3^+) \\
B({\bf X}_3^0, {\bf X}_2^+) & 0 & 0
\end{array}
\right] .
\end{eqnarray*}
Then the equations $(1), (2)$ and $(3)$ are equivalent to
\begin{eqnarray}
C \left[
\begin{array}{c}
c_2^+ \\
c_3^- \\
c_3^+
\end{array}
\right]
=
-A \left[
\begin{array}{c}
\alpha_1 \\
\alpha_2 \\
\alpha_3
\end{array}
\right]
-B \left[
\begin{array}{c}
c_1^- \\
c_1^+ \\
c_2^-
\end{array}
\right] .
\end{eqnarray}
To complete the proof of the lemma, we have only to check that $A, B, C$ are regular matrices.
Indeed, their determinants are not zero;

Notice that $\det{A} = 1 + 2 B({\bf X}_1^0, {\bf X}_2^0) B({\bf X}_2^0, {\bf X}_3^0) B({\bf X}_3^0, {\bf X}_1^0) - B({\bf X}_1^0, {\bf X}_2^0)^2 - B({\bf X}_2^0, {\bf X}_3^0)^2 - B({\bf X}_3^0, {\bf X}_1^0)^2.$
Recall that $B({\bf X}_i^0, {\bf X}_j^0) < -1 \, \, (i \neq j)$. Therefore $\det{A}$ is negative. 

For $B$, we calculate the determinant as follows:
\begin{eqnarray*}
\det{B} &=& B({\bf X}_1^0, {\bf X}_2^-) \{ B({\bf X}_2^0, {\bf X}_1^-) B({\bf X}_3^0, {\bf X}_1^+) - B({\bf X}_2^0, {\bf X}_1^+) B({\bf X}_3^0, {\bf X}_1^-) \} \\
&=& -B({\bf X}_1^0, {\bf X}_2^-) B({\bf X}_2^0 \boxtimes {\bf X}_3^0, {\bf X}_1^- \boxtimes {\bf X}_1^+) \\
&=& B({\bf X}_1^0, {\bf X}_2^-) B({\bf X}_1^-, {\bf X}_1^+) B({\bf X}_2^0 \boxtimes {\bf X}_3^0, {\bf X}_1^0) \\
&=& B({\bf X}_1^0, {\bf X}_2^-) B({\bf X}_1^-, {\bf X}_1^+) \det{({\bf X}_1^0 , {\bf X}_2^0, {\bf X}_3^0)} \neq 0
\end{eqnarray*}
For C, the calculation is similar.
\end{proof}

\begin{lemma}
\label{each pants}
Let ${\bf u}'$ be a map $G_b \to {\mathbb R}_1^2$. If ${\bf u}'$ is a cocycle on each $P_j$, then ${\bf u}$ is a cocycle on $G_b$.
\end{lemma}

\begin{proof}
Since $G_b$ is generated by $g_1, \ldots, g_b$, we must show that the remaining elements are represented as the forms of the cocycle condition of the generators. Note that we have only to show it on $h_j (j \in {\mathbb J})$ and $g_{b+1}$. First consider $h_j$ by an induction.
For $j=1$, ${\bf u}'(h_1) = {\bf u}'(g_2^{-1} g_1^{-1})$ holds. Since ${\bf u}'$ is the cocycle on $P_1$, ${\bf u}'(g_2^{-1} g_1^{-1}) = -g_2^{-1} g_1^{-1} {\bf u}'(g_1) - g_2^{-1} {\bf u}'(g_2)$. We have done when $j=1$.
Second, for $j \geq 2$, we have only to show that 
\begin{eqnarray}
\label{cocycle condition about hj}
{\bf u}'(h_j) = - \sum_{x=1}^{j+1} (g_{j+1}^{-1} \cdots g_x^{-1} {\bf u}'(g_x)) .
\end{eqnarray}
Since ${\bf u}'(h_j) = {\bf u}'(g_{j+1}^{-1} h_{j-1})$, on $P_j$, we obtain ${\bf u}'(h_j)= g_{j+1}^{-1} {\bf u}'(h_{j-1}) - g_{j+1}^{-1} {\bf u}'(g_{j+1})$. By an assumption of the induction, we have the equation $\eqref{cocycle condition about hj}$ for all $h_j$. The case of $g_{b+1}$ is proved in the same way as $h_j$.
\end{proof}


\section{Linear space $D_b$}
\label{Linear space of affine deformations}
To be convenient, we represent $\alpha_i := \alpha(g_i)$ and $\beta_j := \alpha(h_j)$. 
We determine two linear spaces by
\begin{eqnarray}
D_b &:=& \{ ({\bf \alpha}, {\bf \beta}, {\bf t}) \in {\mathbb R}^{3b-3} \, \mid {\bf \alpha} \in {\mathbb R}^{\mathbb I}, {\bf \beta}, {\bf t} \in {\mathbb R}^{\rm J} \}, \\
D_b^0 &:=& \{ ({\bf \alpha}, {\bf \beta}, {\bf 0}) \in D_b \}.
\end{eqnarray}
Here we set ${\bf \alpha} = (\alpha_1, \ldots, \alpha_{b+1})$, ${\bf \beta} = (\beta_1, \ldots, \beta_{b-2})$ and ${\bf t} = (t_1, \ldots, t_{b-2})$.
We call the parameter $t_k$ the {\it affine twist parameter} along $h_k$. It indicates an analogy of hyperbolic geometry, which will be considered geometrically later.

\paragraph{Proof of Theorem \ref{Linear isomorphism}.}
We will construct a linear map $\widetilde{\Phi}: D_b \to {\rm H}^1(G_b, {\mathbb R}_1^2)$, and then we will show that it is a linear isomorphism.
At first, we only consider $D_b^0$, and construct an injective linear map $\widetilde{L}_0:D_b^0 \to {\rm H}^1(G_b, {\mathbb R}_1^2) $;
For any $({\bf \alpha}, {\bf \beta}, {\bf 0}) \in D_b^0$,  we define a cocycle ${\bf u}_0^{\alpha, \beta}$ by the following four steps:
\newline

$(1)$ We define the values of ${\bf u}_0^{\alpha, \beta}$ on the generators of $P_1$. We assign $\alpha_1, \alpha_2, \beta_1$ for their Margulis invariants.
Because of the translation equivalence of cocycles, we can designate the three parameters except the Margulis invariants arbitrarily. Thus we can set, for instance, $c_1^{\pm} = c_2^- =0$. Namely,
\begin{eqnarray*}
{\bf u}_0^{\alpha, \beta}(g_1) &=& \alpha_1 {\bf X}_1^0 + 0 {\bf X}_1^- + 0 {\bf X}_1^+, \\
{\bf u}_0^{\alpha, \beta}(g_2) &=& \alpha_2 {\bf X}_2^0 + 0 {\bf X}_2^- + c_2^+ {\bf X}_2^+, \\
{\bf u}_0^{\alpha, \beta}(h_1) &=& \beta_1 {\bf Y}_1^0 + d_1^- {\bf Y}_1^- + d_1^+ {\bf Y}_1^+.
\end{eqnarray*}
The remaining values $c_2^+, d_1^{\pm}$ are uniquely determined by $\alpha_1, \alpha_2, \beta_1$. Then this correspondence is linear with respect to $\alpha_1, \alpha_2, \beta_1$ from Lemma $\ref{cocycle on pant}$.
\newline

$(2)$ We define the values of ${\bf u}_0^{\alpha, \beta}$ on the generators of $P_2$. $\alpha_3, \beta_2$ are assigned as Margulis invariants.
Since the parameters of ${\bf u}_0^{\alpha, \beta} (h_1)$ are already determined, we must decide one parameter. Hence $c_3^{-}=0$. Namely,
\begin{eqnarray*}
{\bf u}_0^{\alpha, \beta}(h_1^{-1}) &=& -h_1^{-1} (\beta_1 {\bf Y}_1^0 + d_1^- {\bf Y}_1^- + d_1^+ {\bf Y}_1^+), \\
{\bf u}_0^{\alpha, \beta}(g_3) &=& \alpha_3 {\bf X}_3^0 + 0 {\bf X}_3^- + c_3^+ {\bf X}_3^+, \\
{\bf u}_0^{\alpha, \beta}(h_2) &=& \beta_2 {\bf Y}_2^0 + d_2^- {\bf Y}_2^- + d_2^+ {\bf Y}_2^+.
\end{eqnarray*}
The remaining values $c_3^+, d_2^{\pm}$ are unique determined by $d_1^{\pm}, \beta_1, \alpha_3, \beta_2$. Therefore they only depend on $\alpha_1, \alpha_2, \beta_1, \alpha_3, \beta_2$. Then this correspondence is also linear by Lemma $\ref{cocycle on pant}$.
\newline

$(3)$ For $j \in \{ 2 \leq j \leq b-2 \}$, we define ${\bf u}_0^{\alpha, \beta}|_{P_j}$ inductively by the same construction with $(2)$.
\newline

$(4)$ 
We define the values of ${\bf u}_0^{\alpha, \beta}$ on $P_{b-1}$. We decide all parameters by the above-mentioned way. Hence
\begin{eqnarray*}
{\bf u}_0^{\alpha, \beta}(h_{b-2}^{-1}) &=& -h_{b-2}^{-1} (\beta_{b-2} {\bf Y}_{b-2}^0 + d_{b-2}^- {\bf Y}_{b-2}^- + d_{b-2}^+ {\bf Y}_{b-2}^+), \\
{\bf u}_0^{\alpha, \beta}(g_{b}) &=& \alpha_b {\bf X}_b^0 + 0 {\bf X}_b^- + c_b^+ {\bf X}_b^+, \\
{\bf u}_0^{\alpha, \beta}(g_{b+1}) &=& \beta_{b+1} {\bf X}_{b+1}^0 + d_{b+1}^- {\bf X}_{b+1}^- + d_{b+1}^+ {\bf X}_{b+1}^+.
\end{eqnarray*}
The remaining values $c_b^+, c_{b+1}^{\pm}$ are uniquely determined by $\alpha_i, \beta_j$ for $i \in {\mathbb I}, j \in {\mathbb J}$. Then this correspondence is also linear by Lemma $\ref{cocycle on pant}$.
\newline

From the construction, the Lemma $\ref{cocycle on pant}$ and  $\ref{each pants}$, we obtain;
\begin{lemma}
Let the above-mentioned map denoted by $L_0$. Then 
\begin{eqnarray}
\widetilde{L}_0 : D_b^0 \ni ({\bf \alpha}, {\bf \beta}, {\bf 0}) \mapsto [L_0(({\bf \alpha}, {\bf \beta}, {\bf 0}))] \in {\rm H}^1(G_b, {\mathbb R}_1^2)
\end{eqnarray}
is an injective linear map.
\end{lemma}
Next we define special cocycles, which we call {\it affine twists}.
\begin{definition}
An {\rm affine twist} ${\rm AT}_k \, (k \in {\mathbb J})$ is a cocycle which is defined by
\begin{eqnarray*}
{\rm AT}_k |_{P_l} &=& {\bf 0}, \\
{\rm AT}_k |_{P_m} &=& \delta_{{\bf Y}_k^0} |_{P_m},
\end{eqnarray*}
where $l \in \{ 1, \ldots, k \}$ and $m \in \{ k+1, \ldots, b-1 \}$. We note that the affine twist is a well-defined cocycle by Lemma $\ref{each pants}$.
\end{definition}

Finally we will extend the domain of the linear map $L_0$ from $D_b^0$ to $D_b$ by using the affine twists. 
On the other hand, the conclusion of Theorem $\ref{Linear isomorphism}$ comes up with the following proposition.
\begin{proposition}
We define a map $\Phi$ as follows:
\begin{eqnarray*}
\Phi : D_b &\to & {\rm Z}^1(G_b, {\mathbb R}_1^2) \\
({\bf \alpha}, {\bf \beta}, {\bf t}) &\mapsto & {\bf u}_{{\bf t}}^{{\bf \alpha}, {\bf \beta}} := {\bf u}_0^{\alpha, \beta}
+ \sum_{k=1}^{b-2} t_k {\rm AT}_k .
\end{eqnarray*}
Then $\Phi$ descends to a linear isomorphism $\widetilde{\Phi}:D_b \to {\rm H}^1(G_b, {\mathbb R}_1^2)$.
\end{proposition}
\begin{proof}
$\widetilde{\Phi}(D_b) \subset {\rm H}^1(G_b, {\mathbb R}_1^2)$ is trivial. We show that $\widetilde{\Phi}(D_b)$ has $(3b-3)$-dimensions. Set ${\bf u}:= {\bf u}_{{\bf t}}^{\alpha, \beta}$ and suppose that $[{\bf u}]$ equals $[{\bf 0}]$.

We note that all Margulis invariants are zero. Hence ${\bf \alpha} = {\bf \beta} = {\bf 0}$. 
Next we show that $t_1=0$. Since $[{\bf u}]=[{\bf 0}]$, there exists $v \in {\mathbb R}_1^2$ such that ${\bf u} = \delta_v$.
On the pair of pants $P_1$, ${\bf u} (g_1) = {\bf 0}, {\bf u}(g_2) = c_2^+ {\bf X}_2^+, {\bf u}(h_1) = d_1^- {\bf Y}_1^- + d_1^+ {\bf Y}_1^-$ follow. By the cocycle condition, we have $c_2^+ = d_1^{\mp}=0$. Hence ${\bf u}|_{P_1} = {\bf 0}$. So we have $v={\bf 0}$. 
We attend the pair of pants $P_2$. On $P_2$, ${\bf u}(h_1)={\bf 0}, {\bf u}(g_3) = c_3^+ {\bf X}_3^+ + t_1 \delta_{{\bf Y}_1^0} (g_3), {\bf u}(h_2) = d_2^- {\bf Y}_2^- + d_2^+ {\bf Y}_2^+ + t_1 \delta_{{\bf Y}_1^0} (h_2)$. We can easily check the linear independence of ${\bf X}_3^+$ and $\delta_{{\bf Y}_1^0} (g_3)$. Therefore the linear combination $c_3^+ {\bf X}_3^+ + t_1 \delta_{{\bf Y}_1^0} (g_3)$ is zero if and only if $c_3^+= t_1=0$. Then the cocycle condition says that $d_2^{\pm} = 0$. Then we obtain $t_k=0 \, (k \geq 2)$ inductively.
\end{proof}

\paragraph{Remark}
Here we concretely show that all cocycles are defined by our construction up to translation. This observation indicates a geometrical view of the affine twists in the Minkowski space-time. Furthermore this construction corresponds to ours of $\Phi$.

$(1)$ On $P_1$, we determine the value of cocycle of the generators. We decide the Margulis invariants $\alpha_1, \alpha_2, \beta_1$. We assign the translation for $c_1^{\mp}= c_2^- = 0$. These values determine $d_1^{\pm}$. Hence all values on $P_1$.

$(2)$ To determine the values of $P_2$, evaluate the ambiguity.
Then by the cocycle condition, one-dimensional ambiguity still remains.

$(3)$ 
From Charette's works $\cite{c2003} \cite{c2006} \cite{c2009}$, we can note that three invariant axes in $E_1^2$ corresponding to the generators of the pair of pants determines the affine deformation of it.
Now the three invariant axes $C_{\gamma_1}, C_{\gamma_2}, C_{\eta_1}$ are absolutely determined in $E_1^2$. On the other hand, the relativity position among $C_{\eta_1}, C_{\gamma_2}$ and $C_{\eta_2}$ are also determined (namely up to translation). Since there exists the one-ambiguity, it is only to be the direction to $C_{\eta_1}$ (namely ${\bf Y}_1^0$). Note that there exists the position such that $c_3^-=0$. Therefore we decide the ambiguity as the twist parameter. 

$(4)$ On $P_j (j \geq 3)$, we can determine the values of cocycle in the only same way.


\section{Relation to deformations of the hyperbolic structures}
\label{relation to deformations of the hyperbolic structures}
In this section, we will show that an affine twist canonically corresponds to the Fenchel-Nielsen twist along the associated curve. In order to show it, we use a relation between the Lorentzian vector space and a Lie algebra. The discussion is introduced in Goldman and Margulis \cite{gm2000}.

\subsection{Lie algebra and Lorentzian space-time} 
The Lie algebra $sl_2({\mathbb R})$ of ${\rm SL}(2, {\mathbb R})$ has a correspondence with the Lorentzian  space-time ${\mathbb R}_1^2$.
The Lie algebra $sl_2({\mathbb R})$ is a tangent space of $SL(2, {\mathbb R})$ at a unit elements. Furthermore it is isomorphic to a subspace in ${\rm Mat}(2, {\mathbb R})$ of $2 \times 2$ matrices whose element ${\bf X}$ satisfies ${\rm tr}({\bf X})=0$. The Lie algebra $sl_2({\mathbb R})$ is isometric to ${\mathbb R}_1^2$. Indeed,
we take $\{ {\bf e}_1, {\bf e}_2, {\bf e}_3 \}$ as basis.
\begin{eqnarray}
{\bf e}_1 := \left[
\begin{array}{cc}
1 & 0 \\
0 & -1 \\
\end{array}
\right] , {\bf e}_2 := \left[
\begin{array}{cc}
0 & 1 \\
1 & 0 \\
\end{array}
\right] , {\bf e}_3 := \left[
\begin{array}{cc}
0 & -1 \\
1 & 0 \\
\end{array}
\right]
\end{eqnarray}
A {\rm Killing form} is represented by
\begin{eqnarray}
\tilde{B}({\bf X}, {\bf Y}) = \frac{1}{2} {\rm tr}({\bf X} {\bf Y}), 
\end{eqnarray}
where ${\bf X}, {\bf Y} \in sl_2({\mathbb R})$.

\begin{definition}[\cite{gm2000}]
We define a linear map $\psi$
\begin{eqnarray}
\psi : sl_2({\mathbb R}) {\rightarrow} {\mathbb R_1^2} , \, \,  \left[
\begin{array}{cc}
v_1 & v_2 \\
v_3 & -v_1 \\
\end{array}
\right] {\mapsto}  \left[
\begin{array}{c}
v_1 \\
\frac{v_2 + v_3}{2} \\
\frac{-v_2 + v_3}{2} \\
\end{array}
\right]
\end{eqnarray}
\end{definition}
In fact, the linear map $\psi$ is an isomorphism between $sl_2({\mathbb R})$ and ${\mathbb R}_1^2$. Furthermore the Killing form $\widetilde{B}$ is compatible with the Lorentzian inner product $B$ via the isomorphism $\psi$.

\subsection{Margulis invariants and Teichmuller space}
A cocycle can naturally deform the hyperbolic structures after Goldman and Margulis \cite{gm2000}.
We take an eigenvalue $\pm \mu, \, (0<\mu<1)$ of $\widetilde{g} \in SL(2, {\mathbb R})$. We consider the translation length $\ell(\widetilde{g}) = -2 \log{\mu}$.
Under $\psi$, we regard ${\bf u}(g)$ as an element in $sl_2(\mathbb R)$.
Furthermore for $g \in SO^0(2,1)$, we define $\widetilde{g} \in {\rm PSL}(2, {\mathbb R})$, which is corresponding to $g$ under the isomorphism $\psi$. 
Let $\widetilde{\iota}$ be in ${\rm Hom}(G_b, SL(2, {\mathbb R}))$, where
$\widetilde{\iota}(\widetilde{g}) = \widetilde{g} \exp({\bf u}(g)).$ Then we consider the deformation $\widetilde{\iota_t}(\widetilde{g}) = \widetilde{g} \exp(t {\bf u}(g) + O(t^2))$.
The interval of $t$ which can define $\widetilde{\iota}_t(\widetilde{g})$ is denoted by $I_{\widetilde{g}}$. We define the function $L_g^{{\bf u}} : I_{\widetilde{g}} \rightarrow {\mathbb R}, \,\, t \mapsto  \ell (\widetilde{\iota_t}(\widetilde{g})).$ 

\begin{theorem}[\cite{gm2000}]
\begin{eqnarray}
\label{Goldman Margulis}
\alpha_{\bf u}(g) = \frac{1}{2} \frac{d L_g^{{\bf u}}}{dt} (0)
\end{eqnarray}
holds.
\end{theorem}
\subsection{Deformations along affine twists}
We apply the equation $\eqref{Goldman Margulis}$ for the affine twists.
Let $f_l (l \in {\mathbb J})$ denote $g_{l+2}^{-1} g_{l+1}^{-1}$. The translation length $\ell (f_l)$ implies the Fenchel-Nielsen twist along $h_l$. Note that $g_{l+1}$ is in $P_l$ and $g_{l+2}$ is in $P_{l+1}$. See Figure $\ref{twist}$.

\begin{figure}[htbp] 
\begin{center}
\begin{tabular}{c}
\begin{minipage}{0.5\hsize}
\begin{center}
\includegraphics[width=6cm]{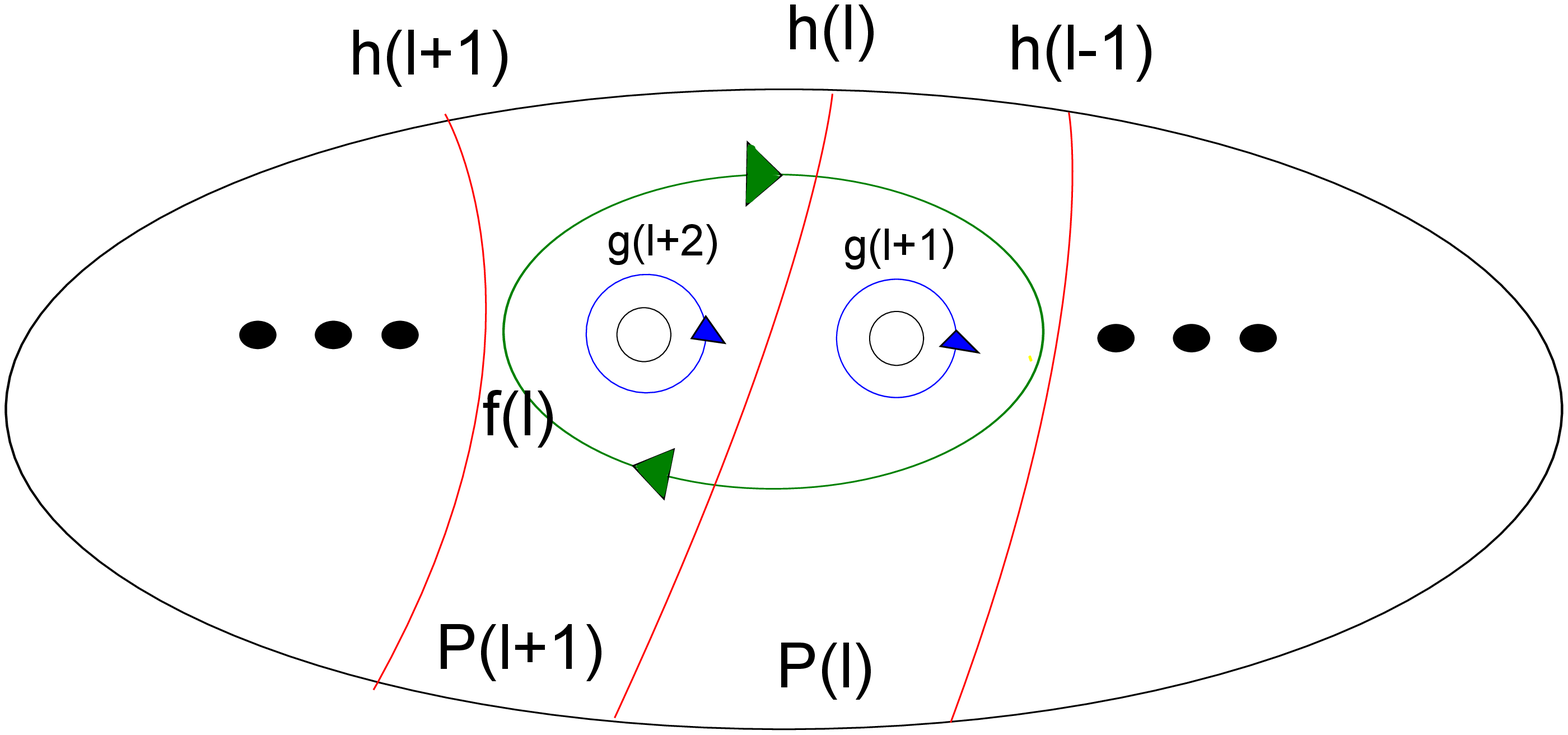}
\hspace{1.6cm}
\end{center}
\end{minipage}
\begin{minipage}{0.33\hsize}
\begin{center}
\includegraphics[width=4cm]{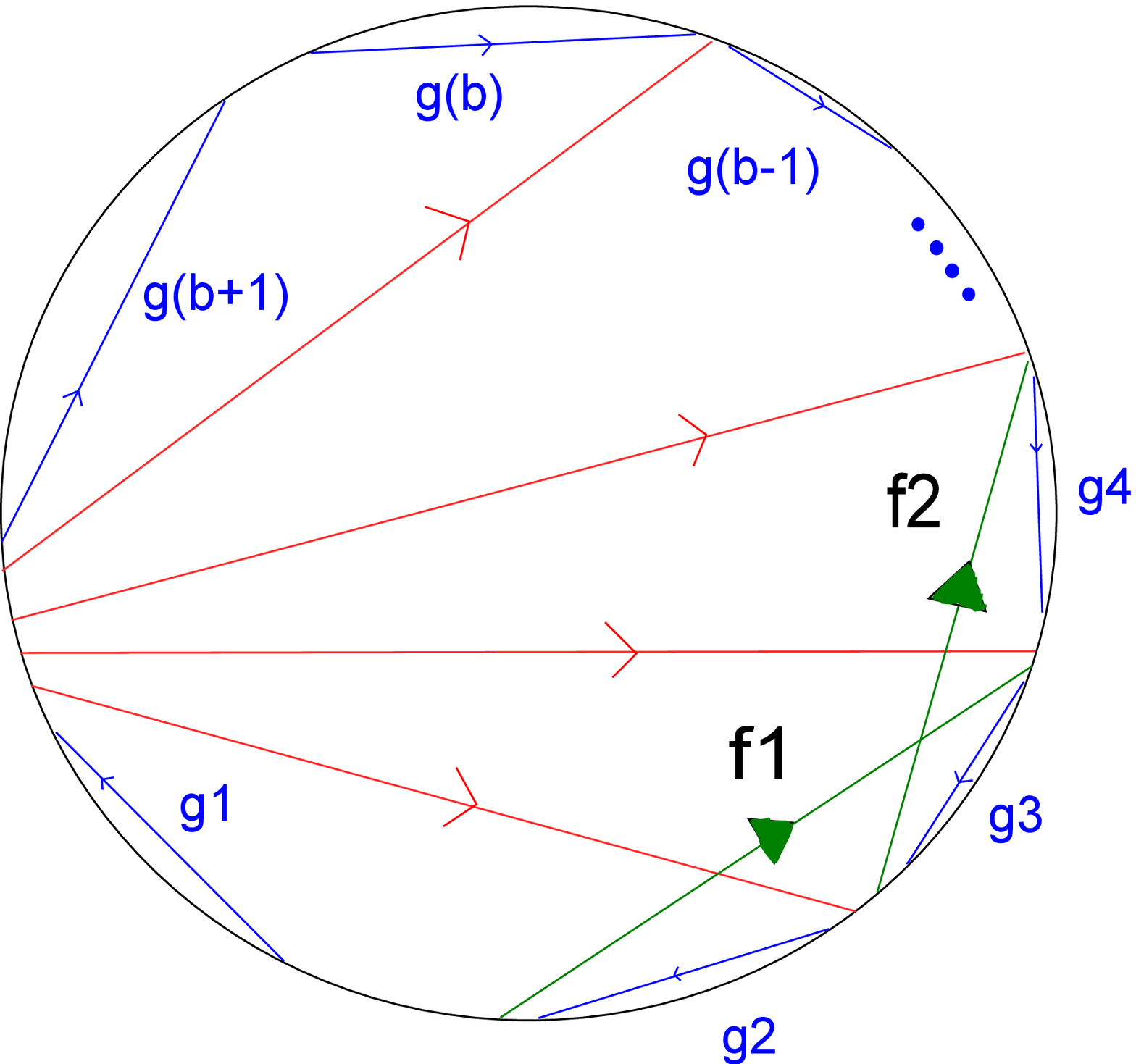}
\hspace{1.6cm}
\end{center}
\end{minipage}
\end{tabular}
\caption{The simple closed curves corresponding to the twists}
\label{twist}
\end{center}
\end{figure} 
We consider a general cocycle ${\bf u} = {\bf u}_0 + \sum_{k=1}^{b-2} t_k {\rm AT}_k$ (${\bf u}_0 \in L_0(D_b^0)$), and calculate the value of $f_l$. By definition, we have ${\bf u}(f_l)={\bf u}_0 (f_l) + \sum_{k=1}^{l} t_k {\rm AT}_k(f_l)$.
We calculate the Margulis invariant $\alpha_{{\bf u}}(f_l)$.
\begin{eqnarray*}
B({\bf u}(f_l), {\bf X}_{f_l}^0) &=& B({\bf u}_0 (f_l) + \sum_{k=1}^{l} t_k {\rm AT}_k(f_l), {\bf X}_{f_l}^0) \\
&=& B({\bf u}_0 (f_l), {\bf X}_{f_l}^0) + t_l B({\rm AT}_{l}(f_l), {\bf X}_{f_l}^0) \\
&=& B({\bf u}_0 (f_l), {\bf X}_{f_l}^0) + t_l B( g_{l+2}^{-1} {\rm AT}_{l}(g_{l+1}^{-1}) + {\rm AT}_{l}(g_{l+2}^{-1}), {\bf X}_{f_l}^0) \\
&=& B({\bf u}_0 (f_l), {\bf X}_{f_l}^0) + t_l B( \delta_{{\bf Y}_1^0}(g_{l+2}^{-1}), {\bf X}_{f_l}^0).
\end{eqnarray*}
The second equation holds because ${\rm AT}_k$ is just the coboundary under restricting on a free product $P_{l} * P_{l+1}$ for $k<l$.
Because $g_{l+2}^{-1} = f_l g_{l+1}$, we can note $B( \delta_{{\bf Y}_1^0}(g_{l+2}^{-1}), {\bf X}_{f_l}^0) = B( \delta_{{\bf Y}_1^0}(g_{l+1}), {\bf X}_{f_l}^0)$.
Let $L^{{\bf u}}_{f_l}(t)$ denote the function which deforms hyperbolic structure of $f_l$ defined by Goldman and Margulis. Then we obtain 
\begin{eqnarray*}
\frac{1}{2} \frac{d L_{f_l}^{{\bf u}}}{dt} (0) = \alpha_{{\bf u}}(f_l) = \alpha_{{\bf u}_0}(f_l) + t_l B( \delta_{{\bf Y}_1^0}(g_{l+1}), {\bf X}_{f_l}^0).
\end{eqnarray*}
The following lemma satisfies the main theorem $\ref{wolpert goldman margulis}
$.
\begin{figure}[htbp] 
\begin{center}
\begin{tabular}{c}
\begin{minipage}{0.5\hsize}
\begin{center}
\includegraphics[width=6cm]{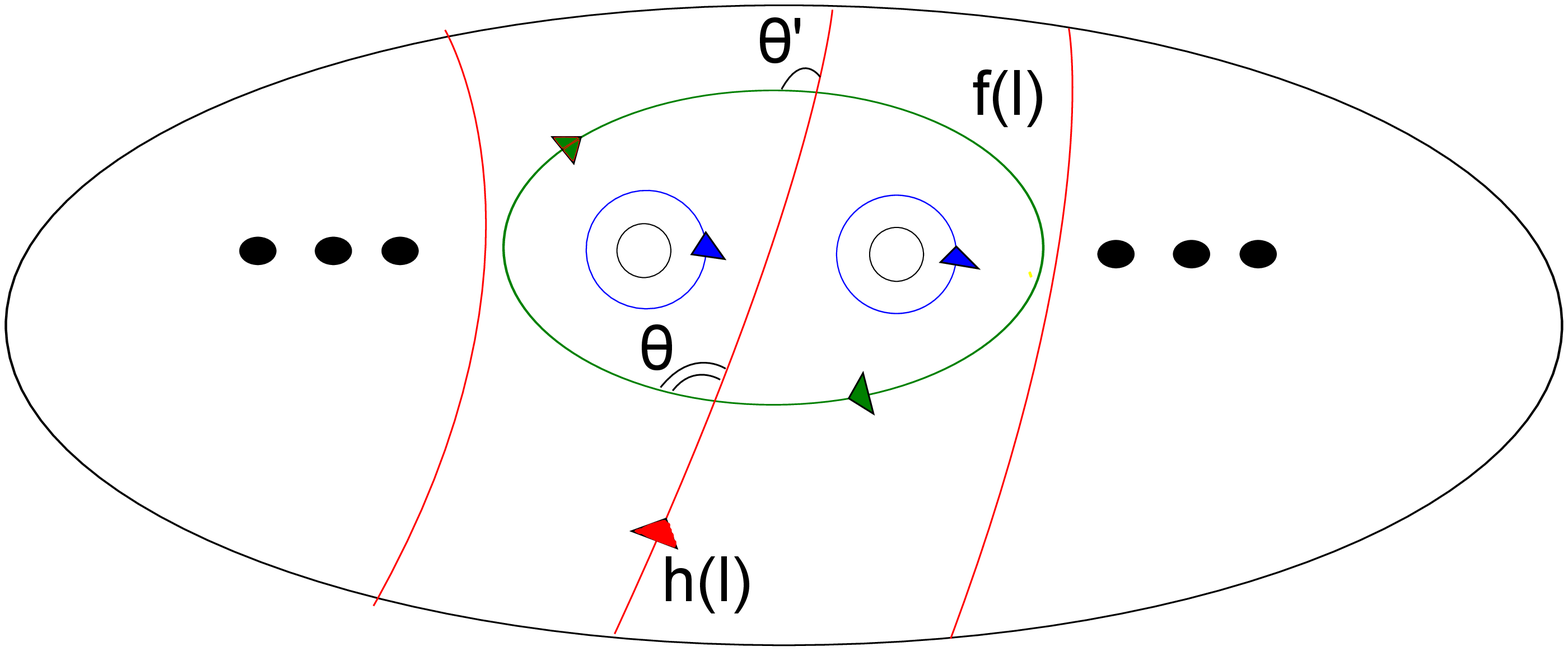} 
\hspace{1.6cm}
\end{center}
\end{minipage}
\begin{minipage}{0.33\hsize}
\begin{center}
\includegraphics[width=4cm]{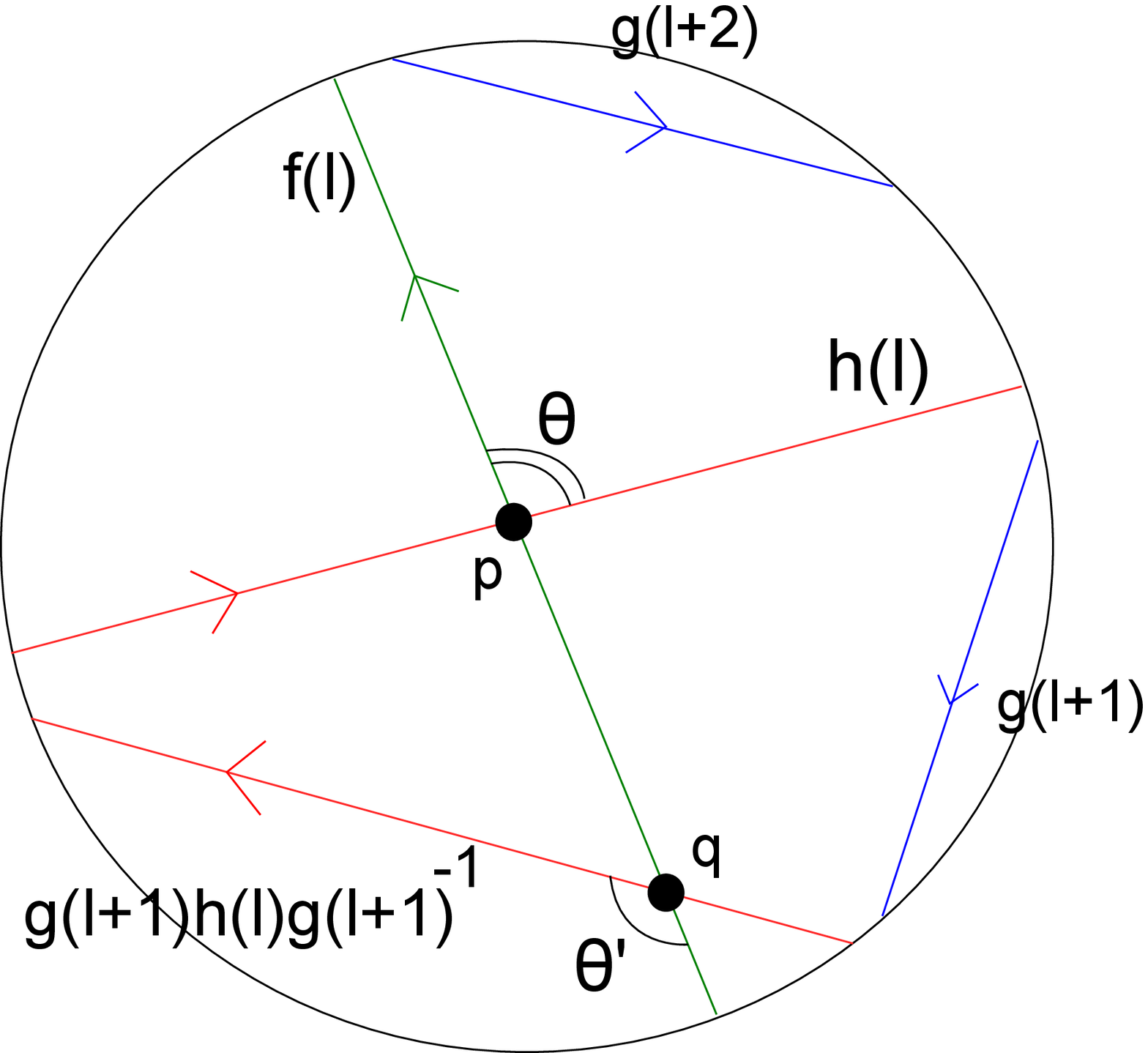} 
\hspace{1.6cm}
\end{center}
\end{minipage}
\end{tabular}
\caption{Angles}
\label{angles} 
\end{center}
\end{figure} 

\begin{lemma}
$$
B( \delta_{{\bf Y}_1^0}(g_{l+1}), {\bf X}_{f_l}^0) = \cos{\theta_l} + \cos{\theta_l'}
$$
holds, where $\theta_l , \theta_l'$ are angles between the closed curves $h_l$ and $f_l$ in $S_{b+1}$. The angles are defined by the segment of the left side of $f_l$ along $h_l$. (Figure $\ref{angles}$.)
\end{lemma}
\begin{proof}
Note that $f_l$ and $h_l$ have the two intersections on $S_{l+1}$. In ${\mathbb H}^2$, one of the points corresponds the intersection of the invariant lines determined by $h_l$ and $f_l$, which we denote by $p$. Another does the intersection of the invariant lines determined by $h_l$ and $g_{l+1} f_l$, which we denote by $q$.
From Lemma $\ref{angle}$ and Figure $\ref{angles}$, we obtain $\theta_k := \cos^{-1}{(B({\bf Y}_l^0, {\bf X}_{f_l}^0))}$ at $p$.
By the definition of the angle, we obtain the angle $\theta_k'$ as $\pi - \cos^{-1}{(B({\bf Y}_l^0, {\bf X}_{f_l}^0))}$.
Then we have
\begin{eqnarray*}
B( \delta_{{\bf Y}_1^0}(g_{l+1}), {\bf X}_{f_l}^0) &=& B( {\bf Y}_1^0, {\bf X}_{f_l}^0) - B(g_{l+1} {\bf Y}_1^0, {\bf X}_{f_l}^0) \\
&=& \cos{\theta_k} - \cos{(\pi - \theta_k')} \\
&=& \cos{\theta_k} + \cos{\theta_k'}.
\end{eqnarray*}
\end{proof}

The formula of Theorem $\ref{wolpert goldman margulis}$ also has a similarity with the one by Wolpert's work in $\cite{w1981}$.
Thus we can regard the affine twist ${\rm AT}_l$ as a correspondence with the Fenchel-Nielsen twist along $h_l$. 

\section{Properly discontinuous action}
\label{properly discontinuous action}
In this section, we give the examples of cocycles in ${\bf Proper}_b$. We only consider the case when $b=3$. However the cases when $b > 3$ are treated in the same manner.
${\bf Proper}_3$ is a subspace of ${\rm H}^1(G_3, {\mathbb R}_1^2)$, whose affine deformations are properly discontinuous on ${E}_1^2$. Note that 
$$
D_3 = \{ (\alpha_1, \alpha_2, \alpha_3, \alpha_4, \beta_1, t_1) \in {\mathbb R}^6 \}.
$$
The author uses disjoint crooked planes corresponding to the generators $\gamma_1, \gamma_2, \gamma_3$. This discussion is applied by the criteria for assigning crooked planes disjointly by Charette, Drumm and Goldman (\cite{cdg2010}\cite{cdg2012}\cite{cdg2014}). Now we consider only {\it positive Margulis invariants} and {\it positive crooked planes}. Thus the author obtains the following equations:
\begin{example}
\label{D3}
We consider six crooked planes $C_i^{\pm} (i=1,2,3)$ such that $\gamma_i (C_i^-) = C_i^+$. Then, {\rm under this construction}, for any ${\rm A}_i, {\rm A}_4^i, {\rm a}_i \in {\mathbb R} \,\, (i=1, 2, 3)$, there exist $b_1, \tau_1 \in {\mathbb R}$ such that 
\begin{eqnarray*}
\alpha_1 &=& {\rm a}_1 - B({\bf X}_4^+, {\bf X}_1^0) {\rm A}_1, \\
\alpha_2 &=& {\rm a}_2 - B(g_3 {\bf X}_4^+, {\bf X}_2^0) {\rm A}_2, \\
\alpha_3 &=& {\rm a}_3 - B({\bf X}_4^+, {\bf X}_3^0) {\rm A}_3, \\
\alpha_4 &=& -B({\bf X}_4^0, {\bf X}_1^+) {\rm A}_4^1 - \lambda_2^{-1} B({\bf X}_4^0, g_1 {\bf X}_2^+) {\rm A}_4^2 -B({\bf X}_4^0, {\bf X}_3^+){\rm A}_4^3, \\
\beta_1 &=& {\rm b}_1 + B({\bf Y}_1^0, {\bf X}_4^+) {\rm A}_1 + B({\bf Y}_1^0, {\bf X}_4^+) \lambda_4 {\rm A}_2, \\
t_1 &=& \epsilon \{ \tau_1 - \lambda_3 B({\bf X}_4^+, {\bf X}_3^+) {\rm A}_3 \}.
\end{eqnarray*}
Here we explain these notations. The real number ${\rm A}_i, {\rm A}_4^i$ mean how far the base points of the two cooked planes $C_i^+, C_i^-$ are. The real number ${\rm a}_i, {\rm b}_1, \tau_1$ are determined by how the six crooked planes are assigned in $E_1^2$. The positive number $\epsilon$ is determined by only a date of the holonomy.
\end{example}
When all of these values equal zero, all crooked planes have a same base point.
We can note that the coefficients of these values are positive. Therefore all crooked planes are pairwise disjoint if and only if ${\rm A}_i, {\rm A}_4^i, {\rm a}_i, {\rm b}_1, {\rm \tau}_1$ are positive.
By Example $\ref{D3}$, the author find a subspace in ${\bf Proper}_3$. 
\begin{example}
For any positive (resp. negative) values ${\bf \alpha} \in {\mathbb R}_+^4 (resp. {\mathbb R}_-^4)$, there exists positive(resp. negative) numbers $b_1 < b'_1 $ (resp.negative) and real numbers $t_1 < t'_1$ such that $\widetilde{\Phi} ( {\bf \alpha} \times (b_1, b'_1) \times (t_1, t'_1) ) \subset {\bf Proper}_3$.

The author does not know for detail what happens in the boundaries of the intervals (namely $b_1, b'_1, t_1$ or $t'_1$). The author guesses two cases that some Margulis invariants become zero, or that more than these intervals can not be obtained under the construction.
\end{example}
However the notation which the author uses is complex and the discussion is so long. Therefore the detail is appeared elsewhere.


\end{document}